\documentclass[leqno,12pt,a4paper]{amsart}%
\usepackage{amsmath,amssymb}
\usepackage{amsfonts}%
\usepackage{color}

\newtheorem{prop}{Proposition}[section]
\newtheorem{teo}{Theorem}[section]
\newtheorem{lema}{Lemma}[section]

\def\R{{\mathbb{R}}}

\def\T{{\mathcal T}}
\def\a{\mathfrak{q}}
\def\L{{\mathcal L}}

\begin{document}

\title[A nonlocal Stefan problem]{\bf A free boundary problem of Stefan type with nonlocal diffusion}

\author[Cort\'{a}zar,  Quir\'{o}s \and Wolanski]{C. Cort\'{a}zar,  F. Quir\'{o}s \and N. Wolanski}

\address{Carmen Cort\'{a}zar\hfill\break\indent
	Departamento  de Matem\'{a}tica, Pontificia Universidad Cat\'{o}lica
	de Chile \hfill\break\indent Santiago, Chile.} \email{{\tt
		ccortaza@mat.puc.cl} }

\address{Fernando Quir\'{o}s\hfill\break\indent
	Departamento  de Matem\'{a}ticas, Universidad Aut\'{o}noma de Madrid
	\hfill\break\indent 28049-Madrid, Spain.} \email{{\tt
		fernando.quiros@uam.es} }

\address{Noem\'{\i} Wolanski \hfill\break\indent
	Departamento  de Matem\'{a}tica, FCEyN,  Universidad de Buenos Aires,
	\hfill\break \indent and
	IMAS, CONICET, \hfill\break\indent
	(1428) Buenos Aires, Argentina.} \email{{\tt wolanski@dm.uba.ar} }

\thanks{C.\,Cort\'azar supported by  FONDECYT grant 1150028 (Chile). F.\,Quir\'os supported by
	projects MTM2014-53037-P  and MTM2017-87596-P (Spain). N.\,Wolanski supported by
	CONICET PIP625, Res. 960/12, ANPCyT PICT-2012-0153, UBACYT X117 and MathAmSud 13MATH03 (Argentina).}



\date{}

\begin{abstract}
We introduce and analyze a nonlocal version of the one-phase Stefan problem in which, as in the classical model, the rate of growth of the volume of the liquid phase is proportional to the rate at which energy is lost through the interphase. We prove existence and uniqueness for the problem posed on the line, and on the half-line with constant Dirichlet data, and in the radial case in several dimensions. We also describe the asymptotic behaviour of both the solution and its free boundary. The model may be of interest to describe the spreading of populations in hostile environments.
\end{abstract}

\maketitle

\section{Introduction}
\setcounter{equation}{0}

The aim of this paper is to introduce and analyze a nonlocal version of the one-phase Stefan problem which may be of interest to describe the spreading of a population surrounded by a hostile environment. As in the classical local formulation, the rate of growth of the volume of the \lq\lq liquid'' phase is proportional to the rate at which energy is lost at the interphase.

The well-known usual local Stefan problem is a mathematical model
that describes the phenomenon of phase transition, for example
between water and ice, \cite{Me}, \cite{Ru}.
Its history goes back to Lam\'{e} and
Clapeyron \cite{LC} and, afterwards, Stefan \cite{St}. The one-phase Stefan
problem corresponds to the simplified case in which the temperature of the ice phase is supposed to be maintained at the value where the phase transition occurs, say
$0^{\circ}{\rm C}$.

Let $\Omega_t$ denote the region occupied by the liquid phase at time $t$. The temperature~$u$ is nonnegative in $\Omega=\{(x,t): x\in\Omega_t,\;t>0\}$, and satisfies the heat equation there. However, the domain occupied by water is not known a priori, and has to be determined at the same time as the temperature. In the classical formulation of the problem, $\Omega$ is assumed to be smooth. As initial data we have the initial location of the liquid phase, $\Omega_0$, and the initial  distribution of  temperature, $u_0$, within it.
Conservation of energy implies that, in the absence of heat sources or sinks, the temperature $u$ satisfies the evolution
equation
$$
\rho c \partial_t u= \nabla \cdot(\kappa \nabla u) \quad\text{in } \Omega,
$$
where the density $\rho>0$,  the 	specific heat $c>0$ (the amount of energy needed to increase in one unit
the temperature of a mass unit of water), and the thermal conductivity
$\kappa>0$ are assumed to be constant. The temperature is expected to be continuous across the \emph{free boundary} $\Gamma=\{(x,t):x\in\partial\Omega_t,\;t>0\}$. Hence, $u=0$ there. However, the domain $\Omega$ is not known a priori, and an extra condition is needed to close the system. This condition, known as \emph{Stefan's condition}, comes also from the conservation of energy, and states that the normal velocity of $\partial\Omega_t$  at any point $x\in\partial_t\Omega$ satisfies
$$
Lv_{\bf n}(x,t)= -\partial_{\bf n} u(x,t),
$$
where the  latent heat $L>0$ (the amount of energy needed to transform a mass unit of ice into water) is also assumed to be constant. All the parameters above, $\rho$, $c$, $\kappa$, and $L$,    can be set to one with a
change of units, and we will assume that this has been done in the discussion that follows.

The one-phase Stefan problem can be used in other contexts, for example in population dynamics. Let us think of a population spreading in a hostile environment. In this setting $u(x,t)$ represents the population density at the point $x$ at time $t$, and $\Omega_t$ denotes the habitat of the population at time $t$. If the population tends to avoid crowds, $u$ will satisfy the heat equation within $\Omega$. However, in the process of colonization of new regions in the hostile environment some individuals will die. It seems sensible to assume that the cost in lifes will be proportional to the volume of the colonized regions. This balance cost/volume should hold at a local level, which leads to Stefan's condition. For the use of Stefan's problem in this context see for example~\cite{Bunting-Du-Krakowski-2012}. We will introduce our nonlocal model having in mind this kind of population spreading problems.


Let $J: \R^N\to \R$ be a nonnegative, radial, continuous function with $\int_{\R^N} J = 1$. Assume also that $J$ is strictly positive in $B(0, d)$ and vanishes in the complement. Let $u(x, t)$  be the density at the point
$x$ at time $t$ of a certain population, and let  $J(x-y)$ be the probability distribution for individuals of jumping
from location $y$ to location $x$. Then, within the viable habitat, the rate at which individuals are arriving at position $x$ from all other places is given by
$$
\mathcal{A}_J u(x,t) := \int_{\mathbb{R}^N}J (x-y)u(y, t)\, dy=(J\ast u(\cdot,t))(x).
$$
Notice that this is nothing but the average of $u(\cdot,t)$ in the ball $B(x,d)$ with weight $J(x-\cdot)$. On the other hand, the rate at which individuals are leaving location $x$ to travel to all other sites is given by $ \int_{\R^N} J (y-x)u(x, t)\, dy=-u(x, t)$. In the absence of external or internal
sources, this leads immediately to
\begin{equation}
\label{eq:nonlocal.heat.equation}
\partial_t u=\mathcal{L}u:=\mathcal{A}_J u -u\quad \text{in }\Omega:=\{(x,t)\in\mathbb{R}^N\times\mathbb{R}_+: x\in\Omega_t,\; t>0\},
\end{equation}
where $\Omega_t\subset\mathbb{R}^N$ is the region apt to be inhabited by the species at time $t$.
Thus, the evolution of the population density at a certain point $x\in\Omega_t$ is given by the balance between its value and its $J$-weighted average  in the ball $B(x,d)$. In the local model the region where the
average is taken shrinks to a point, and  the evolution of the population density is governed by its Laplacian. In the hostile region, $\Omega^c$, which is not apt for the survival of the species, $u=0$.

As data we have the initial habitable region, $\Omega_0$, and the initial distribution of the population, $u_0\in L^1_+(\mathbb{R}^N)\cap C(\mathbb{R}^N)$,  $u_0=0$ in $\Omega_0^c$, where $L^1_+(\mathbb{R}^N):=\{f\in L^1(\mathbb{R}^N):f\ge0\}$.

Individuals that live close to the boundary of $\Omega$ may try to jump, with a probability given by $J$, to the hostile region $\Omega^c$. If this is the case, they will die when crossing the boundary. But their death will not be in vain, since it will prepare the terrain for the arrival of other individuals. Think for example of a region which cannot be inhabited because it has a pH which is inadequate for the species. Individuals that jump there die, but their corpses may change the pH of the sorroundings. In this way, a sound assumption is that the velocity at which the boundary advances in the (outer) normal direction at one of its points is given by the number of  individuals which cross  the boundary through this point in this direction per time and surface units. Thus, for example, in the relatively simple case  in which the problem is posed in one spatial dimension, and the habitable region at time $t$ has the form $\Omega_t=(-\infty,s(t))$ for some $C^1$ function~$s$, which should be nondecreasing, the nonlocal Stefan condition we are looking for reads
\begin{equation}
\label{eq:Stefan.condition.one.boundary}
\dot s(t)=\int_{s(t)}^\infty \mathcal{A}_Ju(\cdot,t).
\end{equation}
The right hand side of this formula thus represents a kind of nonlocal flux at the boundary. We deal with this case in Section~\ref{sect:1D-1FB}, where we prove that the problem is well posed and obtain regularity properties for both the solution and the function $s$ giving the free boundary. We also characterize the large time behaviour of solutions in terms of the initial data, if the latter has a finite first moment.

Again in the one-dimensional setting, if the viable habitat at time $t$ is an interval, $\Omega_t=(s^-(t),s^+(t))$, for a nonincreasing function $s^-$ and a nondecreasing function $s^+$, both of them $C^1$, the nonlocal Stefan conditions are 
$$
\dot s^-(t)=-\int_{-\infty}^{s^-(t)}\mathcal{A}_J u(\cdot,t),\quad  
\dot s^+(t)=\int_{s^+(t)}^{\infty}\mathcal{A}_J u(\cdot,t).
$$
This situation is analyzed in Section~\ref{sect:1D-CS}.

In Section~\ref{sect-semirrecta} we deal with the problem on the half-line with constant Dirichlet data and only one boundary, so that $\Omega_t=(0,s(t))$. Stefan's condition is also given by~\eqref{eq:Stefan.condition.one.boundary} in this case.

The higher dimensional case is more involved. This already happens in the classical local case. Indeed,  for this latter problem it is well-known that even when the initial domain is very smooth, difficulties arise when two points of the boundary of the liquid phase  meet, since the normal direction is not well defined in this situation. Hence, one may only expect local (in time) existence for classical solutions, and weaker notions of solution have to be defined if one looks for global solutions.  This will be done for the present nonlocal model somewhere else. Here we will restrict ourselves to the radial case, for which we can define a global solution; see Section~\ref{sect:radial.solutions}.

Equations like the one appearing in~\eqref{eq:nonlocal.heat.equation} have already been widely used to  model the
dispersal of a species by taking into account long-range effects; see, for example, \cite{BZ, CF, F}. These models usually contain logistic growth terms, to account for births and deaths and for the fact that resources are limited.  However, solutions of such models become immediately positive, and do not have a free boundary.  Our model could also include a growth term. This idea has recently been considered in the work in preparation~\cite{Cao-Du-Li-Li}, of which we have become aware after the completion of the present paper. That work is restricted to the one-dimensional case posed in the whole real line with two free boundaries, which would correspond to the problem that we consider in Section~3. Its main goal is to analyze the possible effects of the growth term in the large time behaviour of solutions. In our case there is no reaction terms, and we are able to give a more precise description of the asympotics.

The possible applications of our model are not restricted to population dynamics. It could be also be meaningful in other contexts, for example to describe phase changes, in order to account for midrange interactions.

In the local case there is another approach to the problem arising from first principles, the so-called enthalpy-temperature formulation. A nonlocal version of  such an approach has been recently analyzed in~\cite{BChQ}. Let us remark that the properties of the solutions of such a model are very different (and less close to the local case) from the properties of the model that we propose here. The problem may exhibit, for example,  mushy regions or nucleation.

We would also like to mention the paper~\cite{cer}, where the authors deal with a nonlinear problem connected to the nonlocal operator $\mathcal{L}$ whose solutions have a free boundary.

\section{The problem on the line with only one free boundary}
\label{sect:1D-1FB} \setcounter{equation}{0}

We start with the simplest case: the problem is posed in one spatial dimension, and the habitable region at time $t$ is assumed to have the form $\Omega_t=(-\infty,s(t))$ for some $C^1$ function~$s$, which should be nondecreasing.

\noindent\emph{Notation. } Along this section $M(t)=\int_{\mathbb{R}}u(\cdot,t)$,  $\Omega_t=\{x\in \mathbb{R}: x<s(t)\}$ for all $t\ge0$ and $\Omega=\{(x,t)\in \mathbb{R}\times\mathbb{R}_+: x\in\Omega_t, t>0\}$.

\medskip

\noindent\textsc{Problem (1D-1FB)}: Given $s_0\in\mathbb{R}$  and $u_0\in L^1_+(\R)\cap C(\mathbb{R})$ such that $u_0(x)=0$ for $x>s_0$, find  a nonnegative function $u\in C(\mathbb{R}\times\overline{\mathbb{R}_+})$ and a nondecreasing function  $s\in C^1(\overline{\mathbb{R}_+})$  satisfying
\begin{equation*}
\label{problem1}
\left\{
\begin{array}{l}
\partial_t u=\mathcal{L}u\ \text{in }\Omega,\quad
u=0\ \text{in }(\mathbb{R}\times\mathbb{R}_+)\setminus\Omega,\quad u(\cdot,0)=u_0\ \text{in }\mathbb{R},\\[8pt]
\displaystyle\dot s(t)=\int_{s(t)}^\infty\mathcal{A}_J u(\cdot,t)\ \text{for }t>0, \quad s(0)=s_0.
\end{array}
\right.
\end{equation*}

\medskip

Observe that in the above description of the problem  we do not require $u(\cdot,t)$ to be positive in $\Omega_t$. In principle there may be regions that are apt to be inhabited in which there is no population. However, we will see later that, if the initial datum is nontrivial, the population will occupy the whole available space for all positive times.  

\subsection{Existence and uniqueness}
Let $(u,s)$ be a solution to Problem~(1D-1FB). Then
\begin{equation}
\label{eq:1D-1FB.integral.version}
\begin{cases}
\displaystyle s(t)=s_0+\int_0^t\int_{s(r)}^\infty\mathcal{A}_J u(x,r)\,\textrm{d}x\textrm{d}r,&t>0,\\[8pt]
\displaystyle u(x,t)=\textrm{e}^{-t}u_0(x)+ \int_{\tau(x)}^t \textrm{e}^{-(t-r)}\mathcal{A}_J u(x,r)\,\textrm{d}r,&\displaystyle x<s_\infty,\; t>\tau(x),\\[8pt]
\displaystyle u(x,t)=0,&\displaystyle x\ge s(t),\;t>0,
\end{cases}
\end{equation}
where $s_\infty=\lim_{t\to\infty}s(t)$ and
$$
\tau(x)=0 \ \text{for }x\le s_0,\quad \tau(x)=
\sup\{t\ge0:s(t)=x\}\ \text{for }x\in (s_0,s_\infty).
$$
On the other hand, if $(u,s)\in C(\overline{\mathbb{R}_+};L^1(\R))\times C(\overline{\mathbb{R}_+})$ satisfies~\eqref{eq:1D-1FB.integral.version}, and $\dot s>0$ in $\mathbb{R}_+$, then $\tau\in C((-\infty,s_\infty))$, and hence $(u,s)$ is a solution to~Problem~(1D-1FB). With this idea in mind, we start by finding a solution to~\eqref{eq:1D-1FB.integral.version}.

\begin{lema}
	\label{lem:existence-line.integral.equation}
	Given $s_0\in\mathbb{R}$  and $u_0\in L^1_+(\R)\cap C(\mathbb{R})$ such that $u_0(x)=0$ for $x>s_0$, there is a unique pair $(u,s)\in C(\overline{\mathbb{R}_+};L^1(\R))\times C(\overline{\mathbb{R}_+})$ solving~\eqref{eq:1D-1FB.integral.version}.
\end{lema}
\begin{proof} We first prove local existence and uniqueness.
	
	Given $T>0$, the linear space $\mathcal{B}_T=C([0,T];L^1(\R))\times C([0,T])$ endowed with the norm
 	$\|(u,s)\|=\|u\|_{L^\infty(0,T;L^1(\R))}+\|s\|_{L^\infty(0,T)}$ is a Banach space. The set
$$
\begin{aligned}
 	K_T=\{(u,s)&\in \mathcal{B}_T: u\ge0,   u(x,t)=0\mbox{ if }x>s(t),  \|u(\cdot,t)\|_{L^1(\R)}\le \|u_0\|_{L^1(\R)}
 \  \forall t>0,\\
 &u(\cdot,0)=u_0, s\mbox{ nondecreasing}, s(0)=s_0\}
 \end{aligned}
$$
is a closed subspace of $\mathcal{B}_T$. Given $(u,s)\in K_T$, we define $(v,\xi)=\mathcal{T}(u,s)$ by
\begin{equation}\label{T}
   \begin{cases}
    \displaystyle\xi(t)=s_0+\int_0^t\int_{s(r)}^\infty\mathcal{A}_J u(x,r)\,\textrm{d}x\textrm{d}r,&0<t\le T,\\
   \displaystyle v(x,t)=\textrm{e}^{-t}u_0(x)+ \int_{\tau_\xi(x)}^t \textrm{e}^{-(t-r)}\mathcal{A}_J u(x,r)\,\textrm{d}r,&\displaystyle x<\xi(T),\ \tau_\xi(x)\le t\le T,\\[8pt]
   \displaystyle v(x,t)=0,&\displaystyle x\ge\xi(t),\;0<t\le T,
      \end{cases}
 \end{equation}
 where
\begin{equation}
\label{eq:def.tau.xi}
 \tau_\xi(x)=0\ \text{for }x\le s_0,\quad \tau_\xi(x)=
 \sup\{t\ge0:\xi(t)=x\}\ \text{for }x\in (s_0,\xi(T))).
\end{equation}

Let us check that $\mathcal{T}(K_T)\subset K_T$. The only conditions that are not trivially verified are the bound $\|v(\cdot,t)\|_{L^1(\R)}\le \|u_0\|_{L^1(\R)}$ and the continuity of $v$ from $[0,T]$ into $L^1(\R)$ with $v(\cdot,0)=u_0$. Let us first see that the bound holds. Indeed, since $\|\mathcal{A}_Ju(\cdot,t)\|_{L^1(\mathbb{R})}=\|u(\cdot,t)\|_{L^1(\mathbb{R})}\le \|u_0\|_{L^1(\mathbb{R})}$, we have that\[\begin{aligned}
\int_{\mathbb{R}} v(\cdot,t)&=\textrm{e}^{-t}\|u_0\|_{L^1(\R)}+\int_{-\infty}^{\xi(t)}\int_{\tau_\xi(x)}^t \textrm{e}^{-(t-r)}\mathcal{A}_J u(x,r)\,\textrm{d}r\textrm{d}x\\
&=\textrm{e}^{-t}\|u_0\|_{L^1(\R)}+\int_0^t \textrm{e}^{-(t-r)}\int_{-\infty}^{\xi(r)}\mathcal{A}_J u(x,r)\,\textrm{d}x\textrm{d}r\le \|u_0\|_{L^1(\R)}.
\end{aligned}
\]

With a very similar computation we see that
\[
\|v(\cdot,t)-u_0\|_{L^1(\R)}\le 2(1-\textrm{e}^{-t})\|u_0\|_{L^1(\R)}\to0\quad\mbox{as } t\to0.
\]

In order to prove the continuity we assume that $t_2>t_1>0$ and, using~\eqref{T},  we get
\[\begin{aligned}
\|v(\cdot,t_1)&-v(\cdot,t_2)\|_{L^1(\R)}\le \big(\textrm{e}^{-t_1}-\textrm{e}^{-t_2}\big)\|u_0\|_{L^1(\R)}\\
&+\int_{-\infty}^{\xi(t_1)}\int_{\tau_\xi(x)}^{t_1}
\big(\textrm{e}^{-(t_1-r)}-\textrm{e}^{-(t_2-r)}\big)\mathcal{A}_J u(x,r)\,\textrm{d}r\textrm{d}x\\
&+\int_{-\infty}^{\xi(t_1)}\int_{t_1}^{t_2} \textrm{e}^{-(t_2-r)}\mathcal{A}_J u(x,r)\,\textrm{d}r\textrm{d}x+\int_{\xi(t_1)}^{\xi(t_2)}\int_{\tau_\xi(x)}^{t_2} \textrm{e}^{-(t_2-r)}\mathcal{A}_J u(x,r)\,\textrm{d}r\textrm{d}x.
\end{aligned}
\]
The monotonicity of $\xi$ and the definition of $\tau_\xi$ imply that $\tau_\xi(x)\ge t_1$ if $x\ge \xi(t_1)$,
and we conclude that
$$
\|v(\cdot,t_1)-v(\cdot,t_2)\|_{L^1(\R)}\le (T+2)(t_2-t_1)\|u_0\|_{L^1(\R)}.
$$

Notice that $(u,s)$ is a  solution to \eqref{eq:1D-1FB.integral.version} for $t\in[0,T]$ if and only if it is a fixed point of $\T:K_T\hookrightarrow K_T$. Let us see that $\T$ is a strict contraction, and hence has a unique fixed point, if $T$ is small enough, how small depending only on $\|J\|_{L^\infty(\R)}$ and $\|u_0\|_{L^1(\R)}$.

Given $(u_i,s_i)\in K_T$, $i=1,2$, let  $(v_i,\xi_i)=\T(u_i,s_i)$.
Since $\|\mathcal{A}_J f(\cdot,t)\|_{L^1(\mathbb{R})}\le \|f(\cdot,t)\|_{L^1(\R)}$ and $\|\mathcal{A}_J f(\cdot,t)\|_{L^\infty(\mathbb{R})}\le \|J\|_{L^\infty(\R)}\|f(\cdot,t)\|_{L^1(\R)}$,
\[
\begin{aligned}
|\xi_1(t)-\xi_2(t)|\le&\left|\int_0^t\int_{s_1(r)}^\infty(\mathcal{A}_J u_1(x,r)-\mathcal{A}_J u_2(x,r))\,{\rm d}x{\rm d}r
\right|\\
&+\left|\int_0^t\int_{s_1(r)}^\infty\mathcal{A}_Ju_2(x,r)\,{\rm d}x{\rm d}r-\int_0^t\int_{s_2(r)}^\infty\mathcal{A}_Ju_2(x,r)\,{\rm d}x{\rm d}r
\right|
\\
\le&\int_0^t\int_{s_1(r)}^\infty(\mathcal{A}_J |u_1-u_2|(x,r)\,{\rm d}x{\rm d}r
+\int_0^t\int_{\min(s_1(r),s_2(r))}^{\max(s_1(r),s_2(r))}\mathcal{A}_Ju_2(x,r)\,{\rm d}x{\rm d}r
\\
\le& T\|u_1-u_2\|_{L^\infty(0,T;L^1(\R))}
+T\|J\|_{L^\infty(\R)}\|u_0\|_{L^1(\R)}\|s_1-s_2\|_{L^\infty(0,T)}.
\end{aligned}
\]
If $\xi_1(t)\le \xi_2(t)$, we get that
$\|v_1(\cdot,t)-v_2(\cdot,t)\|_{L^1(\R)}\le \sum_{i=1}^3 I_i$,
for $0<t<T$,
where
\[
\begin{aligned}
&I_1=\int_{-\infty}^{\xi_1(t)}\int_{\tau_1(x)}^t\textrm{e}^{-(t-r)} \mathcal{A}_J\big|u_1-u_2\big|(x,r)\,\textrm{d}r\textrm{d}x,\\
&I_2=\int_{-\infty}^{\xi_1(t)}\int_{\min(\tau_{\xi_1}(x),\tau_{\xi_2}(x))}^{\max(\tau_{\xi_1}(x),\tau_{\xi_2}(x))} \textrm{e}^{-(t-r)}\mathcal{A}_J u_2(x,r)\,\textrm{d}r\textrm{d}x,\\
&I_3=\int_{\xi_1(t)}^{\xi_2(t)}\int_{\tau_{\xi_2}(x)}^t \textrm{e}^{-(t-r)}\mathcal{A}_J u_2(x,r)\,\textrm{d}r\textrm{d}x.
\end{aligned}
\]
It is readily seen that
\[
I_1\le T\|u_1-u_2\|_{L^\infty((0,T);L^1(\R))}\quad\text{and}\quad
I_3\le  T\|J\|_{L^\infty(\R)}\|u_0\|_{L^1(\R)}|\xi_1(t)-\xi_2(t)|.
\]
On the other hand, since $\xi_1(t)\le \xi_2(t)$ and the functions $\tau_{\xi_1}$ and $\tau_{\xi_2}$ are nondecreasing,
$$
\int_{-\infty}^{\xi_1(t)}|\tau_{\xi_1}(x)-\tau_{\xi_2}(x)|\,\textrm{d}x\le \int_0^t|\xi_1(r)-\xi_2(r)|\,\textrm{d}r,
$$
and we get $I_2\le T\|J\|_{L^\infty(\R)}\|u_0\|_{L^1(\R)}\|\xi_1-\xi_2\|_{L^\infty(0,T)}$.
Thus,  for $t\in[0,T]$,
$$
\|v_1(\cdot,t)-v_2(\cdot,t)\|_{L^1(\R)}\le  T\|u_1-u_2\|_{L^\infty(0,T;L^1(\R))}
+2T\|J\|_{L^\infty(\R)}\|u_0\|_{L^1(\R)}\|\xi_1-\xi_2\|_{L^\infty(0,T)}.
$$
If $\xi_2(t)<\xi_1(t)$, we obtain the same estimate just exchanging the roles of $\xi_1$ and $\xi_2$.

Summarizing,
\begin{equation}
\label{eq:contractivity.1D-1FB}
\|\mathcal{T}(u_1,s_1)-\mathcal{T}(u_2,s_2)\|\le T L\|(u_1,s_1)-(u_2,s_2)\|\quad\text{for all }T\le 1
\end{equation}
for some constant $L>0$ depending only on $\|J\|_{L^\infty(\R)}$ and $\|u_0\|_{L^1(\R)}$.
This gives existence and uniqueness of a fixed point of $\T$ if $T<\min(1/L,1)$.   As $\|u(\cdot,t)\|_{L^1(\R)}$ does not increase, by iterating the procedure we get existence and uniqueness of a fixed point for all $T>0$.
 \end{proof}

We next prove that, if $u_0$ is not trivial,  the solution of problem~\eqref{eq:1D-1FB.integral.version} that we have just constructed is a solution to Problem (1D-1FB).
\begin{prop}
	\label{thm:existence.classical.1D-1FB}
 	Let $(u,s)$ be the unique solution to problem~\eqref{eq:1D-1FB.integral.version} provided by Lemma~\ref{lem:existence-line.integral.equation}. If $u_0\not\equiv0$, then $u>0$ in $\Omega$ and $\dot s>0$ in $\mathbb{R}_+$. Therefore, $(u,s)$ solves Problem~{\rm (1D-1FB)}.
\end{prop}
\begin{proof}
	Let $(x_0,t_0)\in\Omega$ be such that $u(x_0,t_0)=0$. Then
   \[
   0=u(x_0,t_0)\ge \int_{\tau(x_0)}^{t_0}\textrm{e}^{-(t-r)}\mathcal{A}_J u(x_0,r)\,\textrm{d}r\ge0.
   \]
	Hence, $u=0$ in $(x_0-d,x_0+d)\times(\tau(x_0),t_0)$. Iterating this argument, starting with $x_1\in (x_0-d,x_0+d)$, we see that $u(\cdot,t)=0$ for every $t\in(0,t_0]$. Thus, since $u\in C(\overline{\mathbb{R}_+};L^1(\R))$, we would get $u_0\equiv0$, which is a contradiction.
	
	Once we have positivity in $\Omega$, the fact that $s$ is strictly increasing is immediate from the equation for $s$ in~\eqref{eq:1D-1FB.integral.version}.
   \end{proof}

The solution provided by Proposition~\ref{thm:existence.classical.1D-1FB} lies within a class in which there is uniqueness.
\begin{teo}\label{uniqueness-line} Problem {\rm (1D--1FB)} has a unique solution such that $u\in C(\overline{\mathbb{R}_+};L^1(\R))$.
\end{teo}
\begin{proof}
	We have already proved existence if $u_0\not\equiv0$. When $u_0\equiv 0$ we have the trivial solution $u=0$, $s=s_0$.  As for uniqueness, it follows easily from ~\eqref{eq:contractivity.1D-1FB}.
\end{proof}

\subsection{Comparison and regularity}
We have the following strong comparison principle.
\begin{prop}
	\label{prop:comparison} Let $(u,s)$ and $(\widehat u,\widehat s)$ be two solutions to Problem {\rm (1D-1FB)}  with initial data $(u_0,s_0)$ and $(\widehat{u}_0,\widehat{s}_0)$ respectively. 
	Assume that $u_0\ge \widehat{u}_0$ and, either $s_0>\widehat{s}_0$ or $s_0=\widehat{s}_0$ and  $|\{u_0>\widehat{u}_0\}\cap(s_0-d,s_0)|>0$. Then $u>\widehat{u}$ in the  set $\Omega=\{(x,t)\in \mathbb{R}\times\mathbb{R}_+: x<s(t),t>0\}$. 
\end{prop}

\begin{proof}
	The assumptions imply immediately that  there exists $t_0>0$ such that  $s(t)>\widehat s(t)$ for $0<t<t_0$. Then, we deduce that $u>\widehat u$ in the  set $x<s(t)$ for $0<t\le t_0$. Assuming $\bar t=\sup\{\tau>0: s(t)>\widehat s(t)\mbox{ for } 0<t<\tau\}<\infty$ we get a contradiction at time $\bar t$ where $s(\bar t)=\widehat s(\bar t)$ since $u(x,\bar t)>\widehat u(x,\bar t)$ for $x<s(\bar t)$ and the free boundary condition then implies that $\dot s(\bar t)>\dot{\widehat s}(\bar t)$. Therefore, $s(t)>\widehat s(t)$ for every  $t>0$ and we deduce that $u>\widehat u$ in $x<s(t)$ for $t>0$.
 \end{proof}

We now turn our attention to the regularity of the solution.

 \begin{prop}
 	\label {prop:regularity} Let $(u,s)$ be a solution to Problem {\rm (1D-1FB)}. Then, $s\in C^\infty(\mathbb{R}_+)$  and $u\in C^\infty(\{s_0\le x< s_\infty,\ t\ge\tau(x)\})$. On the other hand, in the set
 $\{x\le s_0,\ t>0\}$ the solution $u$ is as smooth as the initial datum $u_0$.
\end{prop}
\begin{proof}
We already know that $\partial_t u\in C(\Omega)$. On the other hand, since $\dot s(t)>0$ for every $t>0$, then $\tau\in C^1((s_0,s_\infty))$. Hence, for $s_0<x<s_\infty$, $t\ge\tau(x)$ there exists
\[
\begin{aligned}
\partial_x u(x,t)&=-\tau'(x) \textrm{e}^{-(t-\tau(x))}\mathcal{A}_J u(x,\tau(x))\\
&+\int_{\tau(x)}^t \textrm{e}^{-(t-r)}\int_{\mathbb{R}}J'(x-y)u(y,r)\,\textrm{d}y\textrm{d}r\in C(\{s_0< x< s_\infty,\ t\ge \tau(x)\}).
\end{aligned}
\]

Since  $\partial_t u(x,t)=0$ for $x>s(t)$, it is easy to check that $\partial_t u\in C(\overline{\mathbb{R}_+};L^1(\R))$.
Now, we go back to the equation for the free boundary and we get that there exists
\[
\ddot s(t)=-\dot s(t)\mathcal{A}_J u(s(t),t)+\int_{s(t)}^\infty\mathcal{A}_J \partial_t u(x,t)\,\textrm{d}x,
\]
so that $s\in C^2(\overline{\mathbb{R}_+})$, and therefore also $\tau\in C^2((s_0,s_\infty))$.
Then, we go back to the formulas for $\partial_x u$ and $\partial_t u$ and we get that $u\in C^2(\{s_0< x<s_\infty,\ t\ge\tau(x)\})$ and $\partial^2_{tt}u\in C(\overline{\mathbb{R}_+};L^1(\R))$.

Iterating this analysis we get the desired regularity result in the region $s_0< x<s_\infty$, $t\ge\tau(x)$.

Finally, in the region $x\le s_0,\ t\ge0$ there holds that $u(x,t)=\textrm{e}^{-t}u_0(x)+h(x,t)$ with
\[
h(x,t)=\int_0^t \textrm{e}^{-(t-r)}\mathcal{A}_J u(x,r)\,\textrm{d}r.
\]
As before, we get that $h\in C^\infty(\{x\le s_0,\ t\ge0\})$ and, since $\textrm{e}^{-t}u_0(x)$ is as smooth as $u_0(x)$ we have the statement of the proposition also in this region.
\end{proof}
Let us remark that solutions are in general only continuous for  $x=s_0$, no matter how smooth  the initial datum is. This is in sharp contrast with the local Stefan problem, for which solutions are $C^\infty$ in $\Omega$. 

\subsection{Asymptotic behaviour} 
Our next aim is to characterize the large time behaviour of the solution to Problem (1D-1FB). The first step is to prove that the rate of growth of the habitable region coincides with the rate of decay of the population, an ingredient which was already present in the modeling. As a consequence, the function $s$ that gives the position of the free boundary  is bounded.
\begin{prop}
	Let $(u,s)$ be a solution to Problem {\rm (1D-1FB)}. Then, $\dot s(t)=-\dot M(t)$, and hence, $s(t)\le M (0)+s_0$.
\end{prop}
\begin{proof} A straightforward computation gives
	\[
	\begin{aligned}
	\dot M(t)&= \int_{-\infty}^{s(t)} \partial_t u(\cdot,t) =  \int_{-\infty}^{s(t)} \mathcal{A}_J u(\cdot,t) - \int_{-\infty}^{s(t)}u(\cdot,t)\\
	&= \int_{-\infty}^{s(t)}\int_{-\infty}^{s(t)}J(x-y)u(y,t)\,\textrm{d}x\textrm{d}y - \int_{-\infty}^{s(t)}u(y,t)\,\textrm{d}y\\
	&=-\int_{-\infty}^{s(t)}\int_{s(t)}^\infty J(x-y)u(y,t)\,\textrm{d}x\textrm{d}y= -\dot s(t).
	\end{aligned}
	\]
	Hence, $s(t)=-M(t)+M(0)+s_0\le M (0)+s_0$.
\end{proof}

In view of this result, we expect $u$ to behave for large times like solutions to the problem on the half-line
\begin{equation}
\label{eq:problem-half-line}
\partial_t v_a-\mathcal{L}v_a=0\quad\mbox{in }(-\infty,a]\times(t_0,\infty),\qquad
v_a=0\quad\mbox{in } (a,\infty)\times(t_0,\infty),
\end{equation}
with $a=s_\infty$ and $t_0$ large.  The special case $a=0$ was studied in~\cite{ceqw2}. All other cases are reduced to it by a traslation, and thus, from the results for $a=0$  we get
$$
\|v_a(\cdot,t)\|_{L^\infty((-\infty,a))}=O(t^{-1}),\qquad \int_{-\infty}^av_a(\cdot,t)=O(t^{-1/2}),
$$
if $v_a$ has a finite first moment at the initial time. Moreover,  
\begin{equation}
\label{eq:convergence.v_a}
\sup_{x<a}\frac {t^{3/2}}{|x|+1}\Big|v_a(x,t)-2M^a_\phi(t_0)\,\frac{\phi(x-a)}{a-x}\mathcal{D}_\a(x-a,t)\Big|\to0\quad\text{as }t\to\infty,
\end{equation}
where 
\begin{equation}\label{eq:phi}
L\phi=0\quad\text{in }\overline{\mathbb{R}}_-, \qquad  \phi=0\quad\text{in } \mathbb{R}_+,\quad |x+\phi(x)|\le C<\infty\quad\text{for }x\in \R_-,
\end{equation}
$M_\phi^a(t_0)=\int_{\mathbb{R}} v_a(x,t_0)\phi(x-a)\,{\rm d}x$, and ${\mathcal D}_\a$ is the so-called dipole solution to the local heat equation with diffusivity $\a=\frac12\int_{\mathbb{R}} J(\xi)\xi^2\,\text{\rm d}\xi$,
$$
{\mathcal D}_\a(x,t)=-\frac x{2\a t}\frac{\textrm{e}^{-\frac{|x|^2}{4\a t}}}{(4\pi\a t)^{1/2}}.
$$

As a first hint that we are on the right track, we prove that,  if the initial data has a finite first moment, then the solution decays  at the same rate as  solutions to the problem on the half-line. An analogous result holds for the mass. As a byproduct, we obtain the limit value of $s$.  
\begin{prop}
	\label{prop:conv.free.boundary.1D-1FB}
	Let $(u,s)$ be a solution to Problem {\rm (1D-1FB)}.
	If $\int_{\mathbb{R}} |x|u_0(x)\,\text{\rm d}x<\infty$, then $\|u(\cdot,t)\|_{L^\infty(\mathbb{R})}=O(t^{-1})$ and  $M(t)=O(t^{-1/2})$. As a consequence of the latter estimate,  $s_\infty=s_0+M(0)$ and  $s_\infty-s(t)=O(t^{-1/2})$ as $t\to\infty$.
\end{prop}
\begin{proof}
	Let $v$ be the solution to problem~\eqref{eq:problem-half-line} with $a=M(0)+s_0$ and $t_0=0$, and initial data $v(\cdot,0)=u_0$.  By the comparison principle, $u\le v$ and the estimates follow from the corresponding estimates for $v$. 
	
	The estimate for the free boundary is immediate from $s_0+M(0)-s(t)=M(t)$.
\end{proof}

We now prove that the asymptotic behaviour is like the one for solutions to the problem posed in the limit support, the half-line $(-\infty,s_\infty)$. 

\begin{prop}
	\label{prop:1D-1FB.convergence} Let $(u,s)$ be a solution to Problem {\rm (1D-1FB)} and $\phi$ as in~\eqref{eq:phi}. If $\int_{\mathbb{R}} u_0(x)|x|\,dx<\infty$, then
	$\lim_{t\to\infty}\int_{\mathbb{R}} u(x,t)(s_\infty-x)\,\text{\rm d}x=M^*\in\mathbb{R}$ and, for every $S<s_\infty$,
	\begin{equation}\label{asym}
	\sup_{x<S}\frac {t^{3/2}}{|x|+1}\Big|u(x,t)-2M^*\frac{\phi(x-s_\infty)}{s_\infty-x}\mathcal{D}_\a(x-s_\infty,t)\Big|\to0\quad\text{as } t\to\infty.
	\end{equation}
\end{prop}
\begin{proof}
	Given $t_0>0$,  let $F(x,t;t_0)$  be the solution to~\eqref{eq:problem-half-line} with $a=s(t_0)$ and initial data at $t=t_0$ given by $u(x,t_0$), and $G(x,t;t_0)$ the solution to the same problem, with the same initial data, but with $a=s_\infty$. A simple comparison argument gives $F(x,t;t_0)\le u(x,t)\le G(x,t;t_0)$ for $t>t_0$. Therefore~\eqref{asym}  will follow from~\eqref{eq:convergence.v_a} if we are able to prove that $M_\phi^-(t):=\int_{\mathbb{R}} u(x,t)\phi(x-s(t))\,{\rm d}x$ and $M_\phi^+(t):=\int_{\mathbb{R}} u(x,t)\phi(x-s_\infty)\,{\rm d}x$ have a common limit $M^*$. Observe that $|M_\phi^-(t)-M_\phi^+(t)|\le C M(t)\le C t^{-1/2}$. Hence it is enough to prove that  $M^*=\lim_{t\to\infty}M_\phi^+(t)$ exists and is finite. 
	
	Since $\phi(x-s_\infty)=\int_{-\infty}^{s_\infty}J(x-y)\phi(y-s_\infty)\,{\rm d}y$,  a simple computation gives
	\[
	\begin{aligned}
	\dot M_\phi^+(t) &=\int_{-\infty}^{s(t)}\Big(\int_{-\infty}^{s(t)} J(x-y)u(y,t)\,{\rm d}y\Big)\phi(x-s_\infty)\,{\rm d}x - \int_{-\infty}^{s(t)} u(x,t)\phi(x-s_\infty)\,{\rm d}x\\
	&=\int_{-\infty}^{s(t)}\Big(\int_{-\infty}^{s(t)}J(y-x)\phi(x-s_\infty){\rm d}x\Big)u(y,t)\,{\rm d}y-\int_{-\infty}^{s(t)}\phi(y-s_\infty)u(y,t)\,{\rm d}y\\
	&=-\int_{s(t)-d}^{s(t)}\int_{s(t)}^{s_\infty}J(x-y)\phi(x-s_\infty)u(y,t)\,{\rm d}x{\rm d}y.
	\end{aligned}
	\]
	Thus, $M_\phi^+$ is nonincreasing, and since $\phi$ is locally bounded, using Proposition~\ref{prop:conv.free.boundary.1D-1FB} we get
	\[
	|\dot M_\phi^*(t)|\le C d\,(s_\infty-s(t))\,t^{-1}\le C t^{-3/2},
	\]
	which implies that $M^*=\lim_{t\to\infty}M_\phi^+(t)$ is finite. 
	
	We finally check that $M^*$  coincides with the asymptotic first moment of $u$ with respect to the point $s_\infty$ (or with respect to any other point).
	Indeed, since $|x+\phi(x)|\le C$ for $x\in\mathbb{R}_-$,
	$$
	|M_\phi^+(t)-\int_{\mathbb{R}} u(x,t)(s_\infty-x)\,{\rm d}x|\le C M(t)\to 0\quad\text{as }t\to\infty.
	$$
\end{proof}

\subsection{Refined asymptotics for the free boundary}
We finally obtain an improved estimate of the asymptotic speed of the free boundary. 
\begin{prop}
	\label{prop:refined.asymptotics.fb}
	Under the assumptions of Proposition~\ref{prop:1D-1FB.convergence},
	\begin{equation}\label{sdot-asym}
	t^{3/2}\,\dot s(t)\to \frac{M^*}{2\sqrt\pi \a^{3/2}}\int_0^d \int_{-d}^0 J(x-y)\phi(y)\,{\rm d}y{\rm d}x.\end{equation}
\end{prop}	
\begin{proof}
	From Proposition~\ref{prop:1D-1FB.convergence} we get $	t^{3/2}u(x,t)\to \frac{M^*}{2\sqrt\pi \a^{3/2}}\phi(x-s_\infty)$ as $t\to\infty$ uniformly on compact sets of $(-\infty,s_\infty)$.
	Moreover,  for every $a>0$, $t^{3/2}u(x,t)$ is  bounded for  $x\in(-a,s_\infty)$ and  $t>0$.
	Hence,
	\[
	\begin{aligned}
	t^{3/2}\,\dot s(t)&=\int_{s(t)}^{s(t)+d}\int_{s(t)-d}^{s(t)} J(x-y)t^{3/2}u(y,t)\,{\rm d}y{\rm d}x\\
	&\to \frac{M^*}{2\sqrt\pi\a^{3/2}}\int_{s_\infty}^{s_\infty+d}\int_{s_\infty-d}^{s_\infty}
	J(x-y)\phi(y-s_\infty)\,{\rm d}y{\rm d}x\\
	&=\frac{M^*}{2\sqrt\pi\a^{3/2}}\int_{0}^{d}\int_{-d}^{0}
	J(x-y)\phi(y)\,{\rm d}y{\rm d}x,
	\end{aligned}
	\]
	so that \eqref{sdot-asym} holds.
\end{proof}

\subsection{Asymptotic behaviour for the corresponding local problem} 
The authors, together with M. Elgueta, showed in~\cite{ceqw2} that solutions to the local problem
\begin{equation*}
\label{eq:local.HL.1D}
\partial_t v-\partial^2_{xx} v=0\ \text{in }\mathbb{R}_-\times\mathbb{R}_+,\quad v(\cdot,0)=v_0(x)\ \text{for }x\in\mathbb{R}_-,\quad v(0,t)=0\ \text{for }t>0 
\end{equation*}
satisfy
$$
\|v(\cdot,t)\|_{L^\infty(\mathbb{R}_-)}=O(t^{-1}),\qquad \int_{\mathbb{R}_-}v(\cdot,t)=O(t^{-1/2}),
$$
if $\int_{\mathbb{R}_-}(1+|x|)v_0(x)\,{\rm d}x<\infty$. Moreover,  
\begin{equation*}
\sup_{\mathbb{R}_-}\frac {t^{3/2}}{|x|+1}\left|v(x,t)+\left(\int_{\mathbb{R}_-}|x|v_0(x)\,{\rm d}x\right)\,\frac x{t}\frac{\textrm{e}^{-\frac{|x|^2}{4 t}}}{(4\pi t)^{1/2}}\right|\to0\quad\text{as }t\to\infty.
\end{equation*}
Hence, using the ideas of the proofs of propositions~\ref{prop:conv.free.boundary.1D-1FB} and \ref{prop:1D-1FB.convergence} we can obtain the asymptotic profile for solutions of the local Stefan problem
\begin{equation}
\label{eq:local.Stef.1D-1FB}
\left\{
\begin{array}{l}
\partial_t u-\partial^2_{xx} u=0\ \text{in }\Omega=\{x<s(t), t>0\},\quad u(\cdot,0)=u_0(x)\ \text{for }x\in\mathbb{R},\\[8pt]
\dot s(t)=-\partial_x u(s(t), t)\ \text{for }t>0, \quad s(0)=s_0.
\end{array}
\right.
\end{equation}
 Since we have not found such results in the literature, we state them here for future reference. 
\begin{teo}
Let $(u,s)$ be the  solution to~\eqref{eq:local.Stef.1D-1FB} with $s_0\in\mathbb{R}$ and $u_0\ge0$  such that $u_0(x)=0$ for $x>s_0$,  $\int_{\mathbb{R}}(1+|x|)u_0(x)\,\text{\rm d}x<\infty$. Then:
	\begin{itemize}
		\item[(i)] $\|u(\cdot,t)\|_{L^\infty(\mathbb{R})}=O(t^{-1})$,  $\int_{\mathbb{R}}u(\cdot,t)=O(t^{-1/2})$;
		\item[(ii)] $s_\infty:=\lim\limits_{t\to\infty}s(t)=s_0+\int_{\mathbb{R}}u_0$, $s_\infty-s(t)=O(t^{-1/2})$ as $t\to\infty$; 
		\item[(iii)] $\lim\limits_{t\to\infty}\int_{\mathbb{R}} u(x,t)(s_\infty-x)\,\text{\rm d}x=:M^*\in\mathbb{R}$;
		\item[(iv)] for every $S<s_\infty$,
		$$
		\sup_{x<S}\frac {t^{3/2}}{|x|+1}\left|u(x,t)-\frac {M^*(s_\infty-x)}{t}\frac{\textrm{e}^{-\frac{|s_\infty-x|^2}{4 t}}}{(4\pi t)^{1/2}}\right|\to0\quad\text{as } t\to\infty.
		$$
	\end{itemize} 
\end{teo}

\section{The problem on the line with compactly supported habitat} 
\label{sect:1D-CS} \setcounter{equation}{0}

We continue our study of the one dimensional case started in the previous section, but now with initial datum of compact support. In this case $\Omega_t=(s^-(t),s^+(t))$ for all $t\ge0$, $\Omega=\{(x,t)\in \mathbb{R}\times\mathbb{R}_+:x\in\Omega_t, t>0\}$ and, again, $M(t)=\int_{\mathbb{R}} u(\cdot,t)$. 

\medskip

\noindent\textsc{Problem (1D-CS)}: Given $\Omega_0=(s^-_0,s^+_0)$ nonempty and bounded, and  $u_0\in C(\mathbb{R})$ nonnegative such that $u_0=0$ in $\mathbb{R}\setminus \Omega_0$,  find  a  nonnegative function $u\in C(\mathbb{R}\times\overline{\mathbb{R}_+})$ and functions  $s^\pm\in C^1(\overline{\mathbb{R}_+})$, $s^-$ nonincreasing and $s^+$ nondecreasing, satisfying
\begin{equation}\label{eq:problem.1D-CS}
\left\{
\begin{array}{l}
\partial_t u-\L u=0\ \text{in }\Omega,\quad u=0 \ \text{in }(\mathbb{R}\times\mathbb{R}_+)\setminus\Omega,\quad
u(\cdot,0)=u_0,\\[8pt]
\dot s^-(t)=-\int_{-\infty}^{s^-(t)}\mathcal{A}_J u(\cdot,t),\  
\dot s^+(t)=\int_{s^+(t)}^{\infty}\mathcal{A}_J u(\cdot,t)\ \text{for } t>0,\ s^\pm(0)=s^\pm_0.
\end{array}
\right.
\end{equation}

\subsection{Existence and uniqueness} Let $(u,s^-,s^+)$ be a solution to Problem~(1D-CS). Then
\begin{equation}
\label{eq:1D-CS.integral.version}
\begin{cases}
\displaystyle s^-(t)=s_0^--\int_0^t\int_{-\infty}^{s^-(r)}\mathcal{A}_J u(x,r)\,\textrm{d}x\textrm{d}r,&t>0,\\	
\displaystyle s^+(t)=s_0^++\int_0^t\int_{s^+(r)}^\infty\mathcal{A}_J u(x,r)\,\textrm{d}x\textrm{d}r,&t>0,\\
\displaystyle u(x,t)=\textrm{e}^{-t}u_0(x)+ \int_{\tau(x)}^t \textrm{e}^{-(t-r)}\mathcal{A}_J u(x,r)\,\textrm{d}r,&\displaystyle x\in (s^-_\infty,s^+_\infty),\; t>\tau(x),\\[8pt]
\displaystyle u(x,t)=0,&\displaystyle x\in\mathbb{R}\setminus (s^-(t),s^+(t)),\;t>0,
\end{cases}
\end{equation}
where $s^\pm_\infty=\lim_{t\to\infty}s^\pm(t)$ and
$$
\tau(x)=\begin{cases}
0& \text{for }x\in[s^-_0,s^+_0],\\[8pt]
\sup\{t\ge0:s^-(t)=x\}\ &\text{for }x\in (s^-_\infty,s^-_0),\\[8pt]
\sup\{t\ge0:s^+(t)=x\}\ &\text{for }x\in (s^+_0,s^+_\infty).
\end{cases}
$$
On the other hand, if $(u,s^-,s^+)\in C(\overline{\mathbb{R}_+};L^1(\R))\times C(\overline{\mathbb{R}_+})\times C(\overline{\mathbb{R}_+})$ solves~\eqref{eq:1D-CS.integral.version}, and $\dot s^-<0$, $\dot s^+>0$ in $\mathbb{R}_+$, then $\tau\in C((s^-_\infty,s^+_\infty))$, and thus $(u,s^-,s^+)$ is a solution to~Problem~(1D-CS). With this idea in mind, we first obtain a solution to~\eqref{eq:1D-CS.integral.version}.

\begin{lema}
	\label{lem:existence1D-CS.integral.equation}
	Given $\Omega_0=(s^-_0,s^+_0)$ nonempty and bounded, and  $u_0\in C(\mathbb{R})$ nonnegative such that $u_0=0$ in $\mathbb{R}\setminus \Omega_0$, there is a unique triple $(u,s^-,s^+)\in C(\overline{\mathbb{R}_+};L^1(\R))\times C(\overline{\mathbb{R}_+})\times C(\overline{\mathbb{R}_+})$ solving~\eqref{eq:1D-CS.integral.version}.
\end{lema}
\begin{proof} We start by proving local existence and uniqueness.
	
	Given $T>0$, the linear space $\mathcal{B}_T=C([0,T];L^1(\R))\times C([0,T])\times C([0,T])$ endowed with the norm
	$\|(u,s^-,s^+)\|=\|u\|_{L^\infty(0,T;L^1(\R))}+\|s^-
	\|_{L^\infty(0,T)}+\|s^+\|_{L^\infty(0,T)}$ is a Banach space. The set
	$$
	\begin{aligned}
	K_T=\{&(u,s^-,s^+)\in \mathcal{B}_T: u\ge0,\   u(x,t)=0\mbox{ if }x\in \mathbb{R}\setminus(s^-(t),s^+(t)),\\ 
	&\|u(\cdot,t)\|_{L^1(\R)}\le \|u_0\|_{L^1(\R)} \mbox{ for all }t>0,\ 
	u(\cdot,0)=u_0,\\ 
	&s^-\mbox{ nonincreasing},\ s^+\mbox{ nondecreasing},\ s^\pm(0)=s^\pm_0\}
	\end{aligned}
	$$
	is a closed subspace of $\mathcal{B}_T$. Given $(u,s^-,s^+)\in K_T$, we define $(v,\xi^-,\xi^+)=\mathcal{T}(u,s^-,s^+)$ by
	\begin{equation*}\label{T.CS}
	\begin{cases}
	\displaystyle \xi^-(t)=s_0^--\int_0^t\int_{-\infty}^{s^-(r)}\mathcal{A}_J u(x,r)\,\textrm{d}x\textrm{d}r,&0<t\le T,\\	
	\displaystyle \xi^+(t)=s_0^++\int_0^t\int_{s^+(r)}^\infty\mathcal{A}_J u(x,r)\,\textrm{d}x\textrm{d}r,&0<t\le T,\\
	\displaystyle v(x,t)=\textrm{e}^{-t}u_0(x)+ \int_{\tau_\xi(x)}^t \textrm{e}^{-(t-r)}\mathcal{A}_J u(x,r)\,\textrm{d}r,&\displaystyle x\in (\xi^-(T),\xi^+(T)),\ \tau_\xi(x)\le t\le T,\\[8pt]
	\displaystyle v(x,t)=0,&\displaystyle x\in\mathbb{R}\setminus(\xi^-(t),\xi^+(t)),\;0<t\le T,
	\end{cases}
	\end{equation*}
	where
	$$
	\tau_\xi(x)=\begin{cases}
	0& \text{for }x\in[s^-_0,s^+_0],\\[8pt]
	\sup\{t\ge0:\xi^-(t)=x\}\ &\text{for }x\in (\xi^-(T),s^-_0),\\[8pt]
	\sup\{t\ge0:\xi^+(t)=x\}\ &\text{for }x\in (s^+_0,\xi^+(T)).
	\end{cases}
	$$
The bound $\|v(\cdot,t)\|_{L^1(\R)}\le \|u_0\|_{L^1(\R)}$ and the continuity of $v$ from $[0,T]$ into $L^1(\R)$ with $v(\cdot,0)=u_0$ are obtained as in the proof of  Lemma~\ref{lem:existence-line.integral.equation}. Hence  $\mathcal{T}(K_T)\subset K_T$. Minor modifications of that proof also allow to show that $v(\cdot,t)$ is continuous in $L^1(\mathbb{R})$ for $t\in(0,T)$, and then that $\mathcal{T}$ is a contraction if $T$ is small enough, how small depending only on $\|u_0\|_{L^1(\mathbb{R})}$. As $\|u(\cdot,t)\|_{L^1(\R)}$ does not increase, by iterating the procedure we get existence and uniqueness of a fixed point for all $T>0$.	
\end{proof}

The same argument given in the proof of Proposition~\ref{thm:existence.classical.1D-1FB} shows that if $u_0$ is not trivial,  the solution $u$ of problem~\eqref{eq:1D-CS.integral.version} that we have just constructed is positive in $\Omega$. Hence $s^\pm$ are strictly monotone, and therefore $(u,s^-,s^+)$ is a solution to Problem (1D-CS). This solution is unique if we stay in the class of solutions that are continuous in $L^1(\mathbb{R})$. 
\begin{teo}Problem {\rm (1D--CS)} has a unique solution such that $u\in C(\overline{\mathbb{R}_+};L^1(\R))$.
\end{teo}
Comparison and regularity results analogous to Propositions~\ref{prop:comparison} and~\ref{prop:regularity} also hold.

\subsection{Asymptotic behaviour} We already know that the functions $s^\pm$ are strictly monotone. We now prove that the rate of growth of the habitable region coincides with the rate of decay of the population. As a consequence, $s^\pm$ are bounded.

\begin{prop} Let the triple $(u,s^-,s^+)$ be a solution to Problem~{\rm (1D-CS)}. Then $\dot M(t)=\dot s^-(t)-\dot s^+(t)$, and $s^\pm$ are bounded. 
\end{prop}
\begin{proof}
	A straightforward computation shows that 
	\[
	\begin{aligned}
	\dot M(t)&=\int_{s^-(t)}^{s^+(t)}\partial_tu(x,t)\,{\rm d}x=\int_{s^-(t)}^{s^+(t)}\mathcal{A}_J u(x,t)\,{\rm d}x-\int_\mathbb{R}u(y,t)\,{\rm d}y\\
	&=-\int_{-\infty}^{s^-(t)}\mathcal{A}_J u(x,t)\,{\rm d}x-\int_{s^+(t)}^\infty\mathcal{A}_Ju(x,t)\,{\rm d}x=\dot s^-(t)-\dot s^+(t).
	\end{aligned}
	\]
	Therefore, $
	\ell(t):=s^+(t)-s^-(t)=s_0^+-s_0^-+M(0)-M(t)\le  s_0^+-s_0^-+M(0)$. Hence, since $\ell$ increases, the limit  $\ell_\infty:=\lim_{t\to\infty} \ell(t)$ exists and is bounded.
	Moreover, for all $t>0$, 
	\[
	s_0^+-\ell_\infty\le s^+(t)-\ell(t)= s^-(t)< s^-_0 <s^+_0< s^+(t)= s^-(t)+\ell(t)\le  s_0^-+\ell_\infty.
	\]
\end{proof}
Comparison from above with the solution of the nonlocal heat equation in the limit domain with zero Dirichlet boundary data gives that the solution decays to 0 exponentially, from where the size of the limit interval follows.
\begin{prop}Let $(u,s^\pm)$ be a solution to Problem~{\rm (1D-CS)}. If $u_0\in L^\infty(\R)$, then $\|u(\cdot,t)\|_{L^\infty(\mathbb{R})}=O(\textrm{e}^{-\lambda t})$  for some $\lambda>0$, and  $
	s^+(t)-s^-(t)\to s_0^+-s_0^-+M(0)$ as $t\to\infty$.
\end{prop}
\begin{proof}
	Let $s^\pm_\infty=\lim_{t\to\infty}s^\pm(t)$, and let $\lambda$ be the first eigenvalue of the operator $\L$ with Dirichlet boundary conditions in the interval $(s^-_{\infty},s^+_{\infty})$. The solution $v$ of $\partial_t v-\L v=0$ in  $I=(s^-_{\infty},s^+_{\infty})$  with $v=0$ in $(\mathbb{R}\setminus I)\times\mathbb{R}_+$ and  $v(\cdot,0)=u_0$ verifies $u(x,t)\le v(x,t)\le C e^{-\lambda t}$
	for some constant $C$. Thus, $
	M(t)\le C e^{-\lambda t}(s_\infty^+-s_\infty-)\to0$ as $t\to\infty$ so that
	$$
	s^+(t)-s^-(t)= s_0^+-s_0^-+M_0-M(t)\to s_0^+-s_0^-+M_0\quad\text{as }t\to\infty.
	$$
\end{proof}
Let us remark that comparison from below with the solution of the nonlocal heat equation in intervals approaching the limit habitat shows that $e^{\mu t}\|u(\cdot,t)\|_{L^\infty(\mathbb{R})}\to\infty$ for all $\mu\in (0,\lambda)$. However, obtaining a sharp rate of decay is a difficult task. 

\section{The free boundary problem on the half-line}\label{sect-semirrecta} \setcounter{equation}{0}

We now consider the problem posed on the half-line, with a constant \lq\lq boundary'' data. The habitable region within $\mathbb{R}_+$ at time $t$ is assumed to have the form $\Omega_t=(0,s(t))$ for some $C^1$ function~$s$, which should be nondecreasing.

\medskip

\noindent\emph{Notation. } Along this section $M(t)= \int_0^{s(t)}u(\cdot,t)$,  $\Omega_t=\{x\in \mathbb{R}_+: x<s(t)\}$ for $t\ge0$, $\Omega=\{(x,t)\in \mathbb{R}_+^2: x\in\Omega_t, t>0\}$.  Notice that  for $x\in\mathbb{R}_+$ the definition of $\mathcal{A}_J u(x,t)$ does not involve the values of $u(x,t)$ for $x<-d$. Hence, while dealing with the problem on the half-line, we denote $\mathcal{A}_J u(x,t)=\int_{-d}^\infty J(x-y)u(y,t)\,{\rm d}y$.

\medskip

\noindent\textsc{Problem}~(HL): Given $s_0\ge0$, $u_0\in L^1_+(\R_+)\cap C(\mathbb{R}_+)$ such that $u_0(x)=0$ for $x>s_0$, and $A\ge 0$, find a nonnegative function $u\in C(\mathbb{R}_+\times\overline{\mathbb{R}_+})$  and a nondecreasing function   $s\in C^1(\overline{\mathbb{R}_+})$  satisfying
\begin{equation}\label{problem2}
\left\{
\begin{array}{l}
\partial_t u-\L u=0\ \text{in }\Omega, \quad
u=0\ \text{in }\mathbb{R}_+^2\setminus\Omega,\quad
u=A\ \text{in } (-d,0)\times\mathbb{R}_+,\\[8pt]
u(\cdot,0)=u_0\ \text{in }\mathbb{R}_+,\quad
\displaystyle \dot s(t)=\int_{s(t)}^\infty\mathcal{A}_Ju(\cdot,t)\ \text{for }t>0,\quad s(0)=s_0.
\end{array}
\right.
\end{equation}

\medskip

Let us remark that  even if  $u_0$ is continuous accross the origin, $u(\cdot, t)$ will have a jump there. That is the reason why we have only asked $u_0$ to be continuous in $\mathbb{R}_+$, since there is no gain in requiring more regularity. 

\subsection{Existence and uniqueness} The  integral version of the problem reads 
\begin{equation}
\label{eq:integral.version.half-line}
\left\{\begin{array}{l}
\displaystyle u(x,t)= \textrm{e}^{-t}u_0(x)+\int_{\tau(x)}^t \textrm{e}^{-(t-r)}\mathcal{A}_Ju(x,r)\,{\rm d}r\mbox{ if }x\in(0,s_\infty), \ t\ge\tau(x),\\[10pt]
u(x,t)=0\mbox{ if } x>s(t),\, t>0,\quad u(x,t)=A \mbox{ if } x\in(-d,0),\, t>0,\\[6pt]
\displaystyle s(t)=s_0+\int_0^t\int_{s(r)}^\infty\mathcal{A}_J u(x,r)\,\textrm{d}x\textrm{d}r\ \text{if }t >0,
\end{array}
\right.
\end{equation}
where $s_\infty=\lim_{t\to\infty}s(t)$ and
$$
\tau(x)=0 \ \text{for }0<x\le s_0,\quad \tau(x)=
\sup\{t\ge0:s(t)=x\}\ \text{for }x\in (s_0,s_\infty).
$$
This latter problem has a unique solution in a suitable functional space.
\begin{lema}
\label{lem:existence-line.integral.equation.half-line}
Given $s_0\ge0$, $u_0\in L^1_+(\R_+)\cap C(\mathbb{R}_+)$ such that $u_0(x)=0$ for $x>s_0$, and $A\ge 0$, there is a unique pair $(u,s)\in C(\overline{\mathbb{R}_+};L^1(-d,\infty))\times C(\overline{\mathbb{R}_+})$ solving~\eqref{eq:integral.version.half-line}.
\end{lema}

\begin{proof} The proof follows the lines of the one of Proposition~\ref{lem:existence-line.integral.equation} for solutions of Problem~(1D-1FB). Hence we only sketch it.
		
For $(u,s)\in K_T$, let $\xi$ and $\tau_\xi$ be defined, respectively, as in~\eqref{T} and~\eqref{eq:def.tau.xi}. Here
\[\begin{aligned}
K_T=
\{(u,s)\in &C([0,T];L^1(-d,\infty))\times C([0,T]): u\ge 0,\ u(x,t)=0\mbox{ in }x>s(t),\\& \|u(\cdot, t)\|_{L^1(\mathbb{R}_+)}\le \|u_0\|_{L^1(\mathbb{R}_+)}+2Ad\mbox{ and }u(x,t)=A\mbox{ in }(-d,0)\ \forall t>0,\\ &\qquad s(0)=s_0,\ s\mbox{ monotone increasing}\}.
\end{aligned}\]
Now, we let
 \begin{equation}\label{T-halfline}
   \begin{cases}
       v(x,t)= \textrm{e}^{-t}u_0(x)+\int_{\tau_\xi(x)}^t \textrm{e}^{-(t-r)}\mathcal{A}_Ju(x,r)\,{\rm d}r&\mbox{if }0<x<\xi(T), \ t\ge\tau_\xi(x),\\
   v(x,t)=0&\mbox{if } x>\xi(t),\\
   v(x,t)=A&\mbox{if } -d<x<0,
      \end{cases}
  \end{equation}
  and we define $\T(u,s):=(v,\xi)$. Then, $\T:K_T\to K_T$ if $T\le \ln 2$. 
  Moreover, $K_T$ is closed in $C([0,T);L^1(-d,\infty))\times C([0,T])$, and $\T$ is a strict contraction in $K_T$ if $T$ is small enough depending only on $\|u_0\|_{L^1(\mathbb{R}_+)}$ and $A$. Therefore, $\T$ has a unique fixed point in $K_T$ and there exists a unique  solution in some maximal time interval $[0,T_0)$.

  Let us see that the maximal solution is global. In fact, assume that $T_0<\infty$. Then, for $t<T_0$,
  \[\begin{aligned}
  \dot M(t)=&\int_0^{s(t)}\partial_t u(x,t)\,{\rm d}x=\int_{-d}^{s(t)}\Big(\int_0^{s(t)}J(x-y)\,{\rm d}x\Big)u(y,t)\,{\rm d}y-\int_0^{s(t)} u(y,t)\,{\rm d}y\\
  =&Ad -A\int_{-d}^0\Big(\int_{-\infty}^0J(x-y)\,dx\Big)\,{\rm d}y-A\int_{-d}^0\Big(\int_{s(t)}^\infty J(x-y)\,{\rm d}x\Big)\,{\rm d}y \\
  &  -\int_0^{s(t)}\Big(\int_{-\infty}^0J(x-y)\,dx\Big)u(y,t)\,{\rm d}y-\int_0^{s(t)}\Big(\int_{s(t)}^\infty J(x-y)\,{\rm d}x\Big)u(y,t)\,{\rm d}y\\
    \le& Ad.
  \end{aligned}
  \]
Hence, $M(t)\le \|u_0\|_{L^1(\mathbb{R}_+)}+AdT_0$ for every  $0<t<T_0$ and therefore, the maximal  solution is defined in $[0,T_0+\delta)$ for some $\delta >0$. This contradicts the definition of $T_0$. Therefore, the solution is global in time.
\end{proof}
Arguing as in Proposition~\ref{thm:existence.classical.1D-1FB} it is easy to see that $u>0$ in $\Omega$ if $u_0\not\equiv0$, and hence that $\dot s>0$. Hence the pair $(u,s)$ given by Lemma~\ref{lem:existence-line.integral.equation.half-line} is a solution to Problem~(HL). This solution is the unique one if we restrict ourselves to functions $u$ that are continuous in $L^1$. 
\begin{teo}\label{existence-halfline} 
	Problem~{\rm (HL)} has a unique solution such that $u\in C(\overline{\mathbb{R}_+};L^1((-d,\infty)))$.	
\end{teo}

\subsection{Comparison and regularity} A comparison principle analogous to Proposition~\ref{prop:comparison} holds. Moreover, the free boundary is $C^\infty$ smooth and $u$ is in $C^\infty(\{s_0\le x\le s(t),\ t\ge0\})$ and as smooth as $u_0$ in the set $\{0\le x\le s_0,\ t\ge0\}$. Since the proofs are similar to the ones we gave in Section~\ref{sect:1D-1FB},  we omit them.

\normalcolor

If the initial datum is bounded, the maximum of the solution is attained at the parabolic boundary.
\begin{prop} 
	\label{prop:weak.maximum.principle}
	Let $(u,s)$ be the solution to problem {\rm (HL)}. 
	If $u_0\in L^\infty(\R_+)$, then
	$$
	\|u(\cdot, t)\|_{L^\infty((-d,\infty))}\le \max\{\|u_0\|_{L^\infty(\mathbb{R}_+)},A\}\quad\mbox{for every } t>0,
	$$
\end{prop}

\begin{proof}
This follows from the integral version of the equation. Indeed, 
$$
u(x,t)\le \text{e}^{-t} u_0(x)+\int_0^t\text{e}^{-(t-s)}\mathcal{A}_J u(x,r)\,{\rm d}r \quad \text{for }x\in \mathbb{R}_+^2,
$$
and hence, denoting $g(t):=\|u(\cdot, t)\|_{L^\infty((-d,\infty)\times(0,t))}$, we have
$$
\|u\|_{L^\infty(\mathbb{R}_+\times(0,t))}\le
\text{e}^{-t} \|u_0\|_{L^\infty(\mathbb{R}_+)}+(1-\text{e}^{-t})g(t)\quad\text{for all }t>0,
$$
Thus, if $g(t)>A$, then $\|u(\cdot,t)\|_{L^\infty(\mathbb{R}_+\times(0,t))}=g(t)$, and from the previous estimate we get $\|u\|_{L^\infty(\mathbb{R}_+\times(0,t))}\le \|u_0\|_{L^\infty(\mathbb{R}_+)}$.
\end{proof}

\subsection{Asymptotic behaviour} The function $s$ giving the position of the free boundary is bounded when $A=0$.
\begin{prop} 
\label{prop:localization}
Let $(u,s)$ be a solution to problem~{\rm (HL)}. If $A=0$, then $s_\infty<\infty$.
\end{prop}
\begin{proof}
Let $M_\psi(t)=\int_0^{s(t)} u(\cdot,t)\psi$, where
$\psi$ is the solution to
\begin{equation}\label{stationary2}
\L \psi=0\quad\mbox{in }\overline{\mathbb{R}_+},\qquad
\psi=0\quad\mbox{in }\mathbb{R}_-,\qquad |\psi(x)-x|\le C<\infty\quad \text{for }x\in \R_+.
\end{equation}
Then, since $u(\cdot,t)=0$ in $[s(t),\infty)$,  $\psi=0$ in $\mathbb{R}_-$, and  $\psi(x)\ge\alpha>0$ in $\overline{\mathbb{R}_+}$, using the equation for the free boundary we get
\[
\begin{aligned}
\dot M_\psi(t)
&=\int_0^{s(t)}\int_{-d}^{s(t)} J(x-y)\psi(x)u(y,t)\,{\rm d}y{\rm d}x
-\int_0^{s(t)}\psi(y)u(y,t)\,{\rm d}y\\
&=-\int_{s(t)}^\infty\int_0^{s(t)} J(x-y)\psi(x)u(y,t)\,{\rm d}y{\rm d}x\le -\alpha \dot s(t).
\end{aligned}
\]
Hence, $\displaystyle s(t)-s_0\le \frac1\alpha(M_\psi(0)-M_\psi(t))\le \frac{M_\psi(0)}\alpha =\frac1\alpha \int_0^{s_0}\psi u_0$.
\end{proof}

As a consequence we have an exponential decay estimate for $u$.
\begin{prop}
Let $(u,s)$ be a solution to problem~{\rm (HL)}. If $A=0$ and $u_0$ is bounded, there exist $\lambda>0$, $C>0$ such that $\|u(\cdot,t)\|_{L^\infty(\mathbb{R})}\le C e^{-\lambda t}$.
\end{prop}
\begin{proof}
Let $v$ be the solution to
\[
	\partial_t v-\L v=0\ \text{in }(0,s_\infty)\times\mathbb{R}_+,\quad
	v=0\ \text{in } \mathbb{R}^2_+\setminus((0,s_\infty)
	\times\mathbb{R}_+),\quad
	v(\cdot,0)=u_0\ \mbox{in }(0,s_\infty).
\]
Then, $0\le u(x,t)\le v(x,t)\le C e^{-\lambda t}$, where $\lambda$ is the first eigenvalue of the operator $\L$  in~$(0,s_\infty)$ with homogeneous Dirichlet boundary conditions; see \cite{ChChR}.
%
\end{proof}
A different situation holds when $A>0$: the population will eventually colonize the whole space.
\begin{prop} Let $(u,s)$ be a solution to problem~{\rm (HL)}. If  $A>0$, then $s_\infty=+\infty$. \end{prop}
\begin{proof}
Let $\psi$ and $M_\psi$ be as in the proof of Proposition~\ref{prop:localization}. Then,
\begin{equation}\label{eq-Mpsi}
\begin{aligned}
\dot M_\psi(t)&=\int_{-d}^{s(t)} u(y,t)\int_0^{s(t)}J(x-y)\psi(x)\,{\rm d}x{\rm d}y-\int_0^{s(t)}\psi(y)u(y,t)\,{\rm d}y\\
&=A\int_{-d}^0\int_0^{s(t)}J(x-y)\psi(x)\,{\rm d}x\,{\rm d}y-\int_0^{s(t)}\int_{s(t)}^\infty J(x-y)\psi(x)u(y,t)\,{\rm d}x\,{\rm d}y\\
&\ge C_0 A-\int_0^{s(t)}\int_{s(t)}^{s(t)+d}J(x-y)\psi(x)u(y,t)\,{\rm d}x{\rm d}y
\end{aligned}
\end{equation}
 for $t\ge 1$ with $C_0=\int_{-d}^0\int_0^{s_1}J(x-y)\psi(x)\,{\rm d}x{\rm d}y>0$ and $s_1=s(1)>0$.

Now, since $\psi(x)\le x+L$ if $x\ge0$ for a certain constant $L$,
\[
\begin{aligned}
\dot M_\psi(t)&\ge C_0 A -(s(t)+d+L)\int_0^{s(t)}\int_{s(t)}^{s(t)+d}J(x-y)u(y,t)\,{\rm d}x{\rm d}y\\
&=C_0 A-(s(t)+d+L)\dot s(t)=\frac{{\rm d}}{{\rm d}t}\big(C_0At-\frac12(s(t)+d+L)^2\big).
\end{aligned}
\]
Hence,
\begin{equation}
\label{eq:estimate.below.s.M.psi}
(s(t)+d+L)^2+2M_\psi(t)\ge 2C_0A(t-1)+(s_1+d+L)^2+2M_\psi(1).
\end{equation}

Assume for a moment that   $u_0$ is bounded. From Proposition~\ref{prop:weak.maximum.principle}, 
if $s_\infty<\infty$, then $M_\psi(t)\le \int_0^{s_\infty}(\psi\, u(\cdot,t))\le s_\infty(s_\infty+L)\max\{\|u_0\|_{L^\infty(0,s_0)},A\}$, 
and we get a contradiction with~\eqref{eq:estimate.below.s.M.psi}, since the right-hand side of the inequality is unbounded.
 
If $u_0$ is not bounded, comparison with the solution corresponding to a truncation of $u_0$  yields the result.
\end{proof}

We now prove that solutions converge to $A$ uniformly on compact sets. As a first step we prove the result for the special case of trivial initial data. 

\begin{lema}
	\label{lem:convergencia.datos.triviales}
Let $(U,S)$ be the solution  to Problem~{\rm (HL)} with initial data $S_0=0$, $U_0=0$. Then $U(\cdot,t)$ converges monotonically to $A$ as $t\to\infty$ and uniformly on compact subsets of $\overline{\mathbb{R}_+}$.
\end{lema}
\begin{proof}
The key point is that $U$ is monotone both in space and time. 

We start with the monotonicity in time. Given $h>0$,  let $v(x,t;h)=U(x,t+h)$, $\zeta(t;h)=S(t+h)$. It is trivial to see that $(v,\zeta)$ is a solution to  Problem~{\rm (HL)}. Since $\zeta(0;h)=S(h)>S(0)$ and $v(x,0;h)=U(x,h)\ge 0=U(x,0)$, comparison yields the desired monotonicity, $v(x,t;h)=U(x,t+h)\ge U(x,t)$.

We now prove that $U(x,t)$ is nonincreasing in $x$ for every $t\ge0$. Given $h>0$, let $T_h:=S^{-1}(h)$, $v(x,t;h):=U(x+h,t)$, and $\zeta(t;h):=S(t)-h$. It is trivial to see that $(v,\zeta)$ satisfies
$$
\left\{
\begin{array}{l}
\partial_t v-\L v=0\ \text{in }\{ x\in(0,\zeta(t)), t>T_h\}, \quad
v=0\ \text{in }\{x\ge \zeta(t), t>T_h\},\\[8pt]
v\le A\ \text{in } (-d,0)\times(T_h,\infty),\quad
\displaystyle \dot \zeta(t)=\int_{\zeta(t)}^\infty\mathcal{A}_Jv(\cdot,t)\ \text{for }t>T_h.
\end{array}
\right.
$$
Moreover, $\zeta(T_h;h)=S(T_h)-h=0<S(T_h)$ and $v(x,T_h;h)=0\le U(x,T_h)$ for $x\in \mathbb{R}_+$. Hence, a comparison argument similar to the one in the proof of Proposition~\ref{prop:comparison} yields $U(x+h,t)=v(x,t;h)\le U(x,t)$ for all $x\in\mathbb{R}_+$ and $t\ge T_h$. On the other hand, $U(x+h,t)=0\le U(x,t)$ if $x\ge0$ and $0\le t\le T_h$, which completes the proof of the monotonicity in space.

We are now ready to prove convergence. Given $x\in\mathbb{R}_+$,  $U(x,t)$ is nondecreasing in $t$ and bounded by $A$. Hence,  there exists $V(x)=\lim_{t\to\infty}U(x,t)\le A$. Even more, 
$\lim_{t\to\infty} \partial_tU(x,t)$ also exists, since $\partial_t U=\mathcal{A}_J U-U$ for $t\ge\tau(x)$ and we can pass to the limit in the convolution by the monotone convergence of $U$ to $V$. We deduce that $\partial_tU(x,t)\to0$ as $t\to\infty$. Therefore, $V$ is a bounded solution to $\L V=0$ in $\R_+$ with $V=A$ on $(-d,0)$, and hence $V\equiv A$.

Let $R>0$. Since $U(R,t)\to A$ and $A\ge U(x,t)\ge U(R,t)$ for $x\le R$, convergence towards $A$ is uniform in $[0,R]$.
\end{proof}

We now pass to the general case, which will follow from a comparison argument.
\begin{prop} Let $(u,s)$ be a solution to problem {\rm (HL)}. Then $u(x,t)\to A$ as $t\to\infty$
	uniformly on compact subsets of $\overline{\R_+}$.
\end{prop}
\begin{proof} 
	Comparison yields $u(x,t)\ge U(x,t)$. Hence,
	$\liminf_{t\to\infty}u(x,t)\ge A$
	uniformly on compact subsets of $\R_+$.

	Let $v$ be the solution to
	\[
	\partial_tv-\L v=0\ \text{in }\R_+^2,\quad v=0\ \text{in } (-d,0)\times\mathbb{R}_+,\quad v(\cdot,0)=(u_0-A)_+ \ \text{in }\R_+.
	\]
	By the results of \cite{ceqw2} we know that $v(x,t)\to0$ uniformly in $\R_+$.
	On the other hand, since $\partial_t(v+A)-\L(v+A)=0$ in $\R_+^2$, $u_0\le v(x,0)+A$  and $v(x,t)+A>A>0$, comparison in $0\le x\le s(t)$, $t>0$ gives $u(x,t)\le v(x,t)+A$, so that
	$ \limsup_{t \to\infty}u(x,t)\le A$ uniformly on compact subsets of $\R_+$.
\end{proof}

\subsection{Refined asymptotics for the free boundary} Now we turn our attention to the asymptotic behaviour of the free boundary. Our aim is to prove that $t^{-1/2}s(t)$ has a limit. As the next lemma shows, this is equivalent to showing that $t^{-1}M_\psi(t)$ converges.
\begin{lema}
	\label{lem:equivalence}
	Let $(u,s)$ be a solution to problem {\rm (HL)} with  $A>0$.
	If $F(t):=t^{-1}M_\psi(t)$ converges to $F_\infty$ as $t\to\infty$,  then $
	\lim_{t\to\infty}t^{-1/2}s(t)=(2C_1 A-2F_\infty)^{1/2}$.
\end{lema}
\begin{proof}
	Going back to \eqref{eq-Mpsi} and using that $s(t)\to\infty$ we find for some $t_0$ large that
	\[
	M_\psi(t)\ge C_1A(t-t_0)+\frac12(s(t_0)+d+L)^2-\frac12(s(t)+d+L)^2+M_\psi(t_0)\quad \text{for }t\ge t_0,
	\]
	where $C_1=\int_{-d}^0\int_0^d J(x-y)\psi(x)\,{\rm d}x{\rm d}y$.
	Again from \eqref{eq-Mpsi} and using this time that $\psi(x)\ge x-L$ and $u(x,t)\le A$ we get for every $t>0$,
	\[\begin{aligned}
	\dot M_\psi(t)&=A\int_{-d}^0\int_0^{s(t)}J(x-y)\psi(x)\,{\rm d}x{\rm d}y
	-\int_0^{s(t)}\int_{s(t)}^\infty J(x-y)\psi(x)u(y,t)\,{\rm d}x{\rm d}y\\
	&\le A\int_{-d}^0\int_0^d J(x-y)\psi(x)\,{\rm d}x{\rm d}y
	-(s(t)-L)\int_0^{s(t)}\int_{s(t)}^\infty J(x-y)u(y,t)\,{\rm d}x{\rm d}y\\
	&=C_1 A-(s(t)-L)\dot s(t)=\frac d{dt}\big[C_1At -\frac12(s(t)-L)^2\big].
	\end{aligned}
	\]
	Thus, $M_\psi(t)\le  C_1A t-\frac12 (s(t)-L)^2+\frac12L^2$,
	and hence	$\left|t^{-1}(s(t))^2-\left(2C_1A-2F(t)\right)\right|=o(1)$ as $t\to\infty$.
\end{proof}

We start by considering the special case of the solution $(U,S)$ with trivial initial data.
\begin{lema}
	\label{prop:conv.s.trivial.initial.data}
	Let $(U,S)$ be as in~Lemma~\ref{lem:convergencia.datos.triviales}. Then $F(t)$ converges. Hence, there is a constant $c_*>0$ such that $t^{-1/2}S(t)\to c_*$ as $t\to\infty$. 
\end{lema} 
\begin{proof}
From~\eqref{eq-Mpsi} we get
\[\begin{aligned}
\dot M_\psi(t)&\le C_1A-\int_0^{S(t)}\int_{S(t)}^\infty J(x-y)\psi(x)U(y,t)\,{\rm d}x{\rm d}y\le C_1A.
\end{aligned}
\]
Hence, $
M_\psi(t)\le C_1At$, which means that $F$ is bounded. On the other hand, 
\[\begin{aligned}
F'(t)=&\frac1{t^2}\left(t\int_0^{S(t)}\int_{-d}^\infty J(x-y)\psi(x)U(y,t)\,{\rm d}y{\rm d}x\right.\\
&\left.-\int_{\tau(x)}^te^{-(t-r)}\int_0^{S(t)}\int_{-d}^\infty J(x-y)\psi(x)U(y,r)\,{\rm d}y{\rm d}x{\rm d}r\right).
\end{aligned}
\]
Since  $U(y,r)\le U(y,t)$ for $r\le t$, $e^{-(t-r)}\le 1$ and $t-\tau(x)\le t$, see the proof of~Lemma~\ref{lem:convergencia.datos.triviales},
and hence $F'\ge0$. We conclude that there exists $F_\infty=\lim_{t\to\infty} F(t)$, from where the existence of $c_*$ follows, thanks to Lemma~\ref{lem:equivalence}.
\end{proof}

Now we consider more general data. Unfortunately, we have to impose a technical restriction on the size of the initial data. 

\begin{prop} Let $(u,s)$ be a solution to problem {\rm (HL)}. If $A>0$ and $\|u_0\|_{L^\infty(0,s_0)}<A$, then 
$t^{-1/2}s(t)\to c_*$ as $t\to\infty$, where $c_*$ is the constant given by Proposition~\ref{prop:conv.s.trivial.initial.data}.
\end{prop}
\begin{proof}
Lemma~\ref{lem:convergencia.datos.triviales} implies that there exists $t_0>0$ such that
$
U(x,t_0)\ge \|u_0\|_{L^\infty}\ge u_0(x)$ in $0\le x\le s_0$.
Moreover, we can choose $t_0$ large enough so that $S(t_0)>s_0$. Thus,
by the comparison principle for solutions of Problem (HL), we get
 $U(x,t)\le u(x,t)\le U(x,t+t_0)$ for $(x,t)\in\mathbb{R}_+^2$.
Therefore,
\[
\frac{\int_{0}^{S(t)} (\psi U(\cdot,t))}t\le F(t)=\frac{\int_0^{s(t)} (\psi u(\cdot,t))}t\le\frac{\int_0^{S(t+t_0)} (\psi U(\cdot,t+t_0))}t.
\]
Since both the left-hand side and the right-hand side converge to the same constant $F_\infty$, see Lemma~\ref{prop:conv.s.trivial.initial.data}, we conclude that $F(t)$ converges to $F_\infty$, from where the result for the asymptotic behaviour of the free boundary follows immediately.
\end{proof}

\section{Radial solutions in higher dimensions}
\label{sect:radial.solutions} \setcounter{equation}{0}

In this section we deal with radial solutions in the whole space in any spatial dimension. 

\noindent\emph{Notation. }  Let $f$ be a radial function, $f(x)=f_0(r)$, $r=|x|$. If no confusion arises we will use the same symbol $f$ both for the original function $f$ and for its radial version $f_0$.

Let us recall that if $f$ is a radial function, its Fourier transform $\mathcal{F}f$ is also a radial function, 
$$
\mathcal{F}f(|\xi|)=\frac{2\pi}{|\xi|^{\frac N2-1}}\int_0^\infty f(r)J_{\frac N2-1}(2\pi r|\xi|)r^{\frac N2}\,{\rm d}r,
$$
where $J_\nu$  denotes the Bessel functions of the first kind of order $\nu$; see~Theorem IV.3.3 in~\cite{Stein-Weiss-book}. Therefore, the convolution of two radial functions 
$$
(J*f)(x)=\int_{\mathbb{R}^N}J(|x-y|)f(|y|)\,{\rm d}y=\int_{\mathbb{R}^N}\mathcal{F}J(\xi)\mathcal{F}f(\xi)\textrm{e}^{2\pi \textrm{i}x\cdot\xi}\,{\rm d}\xi
$$
is also radial, and can be expressed by the one-dimensional integral 
$$
(J*f)(x)=\frac{2\pi}{|x|^{\frac N2-1}}\int_0^\infty \mathcal{F}J(r)\mathcal{F}f(r)J_{\frac N2-1}(2\pi r|x|)r^{\frac N2}\,{\rm d}r.
$$

\noindent\emph{Notation. } The  measure of the unit ball in $\mathbb{R}^N$ will be denoted by $\omega_N$, the population at time $t\ge0$ by $M(t)=\int_{\mathbb{R}^N}u(\cdot,t)$, and its viable habitat by $\Omega_t=\{x\in \mathbb{R}^N:|x|<R(t)\}$. Finally,  $\Omega:=\{(x,t)\in\mathbb{R}^N\times \mathbb{R}_+: x\in\Omega_t,\, t>0\}$, and $B_r=B(0,r)$..

The problem looks similar to problem (1D-1FB).

\medskip

\noindent\textsc{Problem (R)}: Given $R_0>0$ and $u_0\in C(\mathbb{R}^N)$ nonnegative and radially symmetric such that $u_0=0$ in $\mathbb{R}^N\setminus B_{R_0}$,  find  a nonnegative function $u\in C(\mathbb{R}^N\times\overline{\mathbb{R}_+})$, radially symmetric in the spatial variable,  and a nonincreasing function  $R\in C^1(\overline{\mathbb{R}_+})$ such that
\begin{equation}\label{eq:problem.R}
\left\{
\begin{array}{l}
\partial_t u-\L u=0\ \text{in }\Omega,\quad u=0 \ \text{in }(\mathbb{R}^N\times\mathbb{R}_+)\setminus\Omega,\quad
u(\cdot,0)=u_0,\\[8pt]
\displaystyle\frac{\rm d}{{\rm d}t}(R^N(t))=N\int_{R(t)}^{\infty}r^{N-1}\mathcal{A}_J u(r,t)\,{\rm d}r\ \text{for } t>0,\quad R(0)=R_0.
\end{array}
\right.
\end{equation}

\subsection{Existence and uniqueness}The integral version of the problem reads
\begin{equation}
\label{eq:radial.integral.version}
\begin{cases}
\displaystyle R^N(t)=R_0^N+N\int_0^t\int_{R(z)}^\infty r^{N-1}\mathcal{A}_J u(r,z)\,\textrm{d}r\textrm{d}z,&t>0,\\[8pt]
\displaystyle u(r,t)=\textrm{e}^{-t}u_0(r)+ \int_{\tau(r)}^t \textrm{e}^{-(t-z)}\mathcal{A}_J u(r,z)\,\textrm{d}z,&\displaystyle r<R_\infty,\; t>\tau(r),\\[8pt]
\displaystyle u(r,t)=0,&\displaystyle r\ge R(t),\;t>0,
\end{cases}
\end{equation}
where $R_\infty=\lim_{t\to\infty}R(t)$ and
$$
\tau(r)=0 \ \text{for }r\le R_0,\quad \tau(r)=
\sup\{t\ge0:R(t)=r\}\ \text{for }r\in (R_0,R_\infty).
$$
This problem has a unique solution in the appropriate functional space. 

\begin{lema}
	\label{lem:radial.integral.equation}
	Given $R_0>0$  and $u_0\in  C(\mathbb{R}^N)$ nonnegative and radially symmetric such that $u_0=0$ in the set $\mathbb{R}^N\setminus B_{R_0}$,   there is a unique pair $(u,R)\in C(\overline{\mathbb{R}_+};L^1(\R^N))\times C(\overline{\mathbb{R}_+})$, with $u$ radially symmetric solving~\eqref{eq:radial.integral.version}.
\end{lema}

The proof is essentially the same one as that of Lemma~\ref{lem:existence-line.integral.equation}, with the obvious changes to take into the account the radial symmetry and the weight $r^{N-1}$.
It is then easy to see that if  $u_0$ is not trivial,  the solution $u$ of problem~\eqref{eq:radial.integral.version} that we have just constructed is positive in $\Omega$. Hence $R$ is strictly monotone, and therefore $(u,R)$ is a solution to Problem~(R). This solution is unique if we stay in the class of solutions that are continuous in $L^1(\mathbb{R}^N)$. 
\begin{teo}Problem {\rm (R)} has a unique solution such that $u\in C(\overline{\mathbb{R}_+};L^1(\R^N))$.
\end{teo}
Comparison and regularity results analogous to Propositions~\ref{prop:comparison} and~\ref{prop:regularity} also hold.

\subsection{Asymptotic behaviour} As expected, the rate of growth of the volume of the habitable region coincides with the rate at which the total population decreases. As a consequence, the habitat stays confined in a bounded ball.

\begin{prop}
Let $(u,R)$ be a solution to Problem {\rm (R)}.  Then, 
$
\frac{\rm d}{{\rm d}t}|\Omega_t|=-\frac{\rm d}{{\rm d}t}\int_{\mathbb{R}^N} u(\cdot,t)$.
Hence  $R(t)\le \left(\frac{M(0)}{\omega_N}+R_0^N\right)^{1/N}$.
\end{prop}
\begin{proof} 
 A straightforward computation shows that
	\[
	\begin{aligned}
	\dot M(t)&=\int_{\{|x|<R(t)\}} \partial_t u(\cdot,t)=\int_{\{|x|<R(t)\}}\mathcal{A}_Ju(\cdot,t)-\int_{\mathbb{R}^N} u(\cdot,t)=-\int_{\{|x|>R(t)\}}\mathcal{A}_J u(x,t)\,{\rm d}x\\
		&=-N\omega_N\int_{R(t)}^{\infty}r^{N-1}\mathcal{A}_J u(r,t)\,{\rm d}r=-\frac{\rm d}{{\rm d}t}(\omega_NR^N(t))=-\frac{\rm d}{{\rm d}t}|\Omega_t|.
	\end{aligned}
	\]
Hence $|\Omega_t|\le M(t)+	|\Omega_t|=M(0)+|\Omega_0|$, from where the bound for $R(t)$ follows immediately.
\end{proof}
As a corollary, we get the exponential decay to 0 of the solution, and the limit habitat.
\begin{prop} Let $(u,R)$ be a solution to Problem~{\rm (R)}.  There are constants $C,\lambda>0$ such that $u(x,t)\le C e^{-\lambda t}$. As a consequence, $R(t)\to R_\infty:=\left(\frac{M(0)}{\omega_N}+R_0^N\right)^{1/N}$,
\end{prop}
\begin{proof}  Since $R(t)\le R_\infty$, we have that  $u\le v$, where $v$ is the solution to
	\[
	\partial_t v-\L v=0\ \text{in }B_{R_\infty}\times\mathbb{R}_+,\quad
	v=0\ \text{in }(\R^N\setminus B_{R_\infty})\times\mathbb{R}_+,\quad
	v(\cdot,0)=u_0 \ \text{in }\mathbb{R}^N.
	\]
	Therefore, since $u_0\in L^\infty(\R^N)$,  $u(x,t)\le v(x,t)\le C e^{-\lambda t}$ where $\lambda>0$ is the first eigenvalue of the operator $-\L$ with Dirichlet conditions, in $B_{R_\infty}$. Finally, this bound gives that $M(t)\to0$ as $t\to\infty$ implying that $R(t)\to R_\infty$.
\end{proof}
%
%
%
%
%
%
%
%


\end{document}